\documentclass{article}
\usepackage[utf8]{inputenc}
\usepackage{amsmath, amsthm, amssymb, amsfonts}
\usepackage{bm}
\usepackage{dsfont}
\usepackage[utf8]{inputenc}
\usepackage{rotate}
\usepackage{tikz}
\usepackage{tikz-cd}
\usepackage[arrow,matrix,curve]{xy} 	
\usepackage{xcolor}
\usepackage{todonotes}
\usepackage{alltt} 
\usepackage{amsmath,amssymb}

\theoremstyle{plain}
\newtheorem{theorem}{Theorem}[section]
\newtheorem{proposition}[theorem]{Proposition}

\newtheorem{corollary}[theorem]{Corollary}

\theoremstyle{definition}
\newtheorem{definition}[theorem]{Definition}

\newtheorem{remark}[theorem]{Remark}

\title{Polynomial p-adic Low-Discrepancy Sequences}
\author{Christian Weiss}
\date{\today}

\begin{document}

\maketitle

\begin{abstract} The classic example of a low-discrepancy sequence in $\mathbb{Z}_p$ is $(x_n) = an+b$ with $a \in \mathbb{Z}_p^x$ and $b \in \mathbb{Z}_p$. Here we address the non-linear case and show that a polynomial $f$ generates a low-discrepancy sequence in $\mathbb{Z}_p$ if and only if it is a permutation polynomial $\mod p$ and $\mod p^2$. By this it is possible to construct non-linear examples of low-discrepancy sequences in $\mathbb{Z}_p$ for all primes $p$. Moreover, we prove a criterion which decides for any given polynomial in $\mathbb{Z}_p$ with $p \in \left\{ 3,5, 7\right\}$ if it generates a low-discrepancy sequence. We also discuss connections to the theories of Poissonian pair correlations and real discrepancy. 
\end{abstract}

\section{Introduction}
While uniform distribution theory is in the real setting typically formulated for sequences in the unit interval $[0,1]$, its p-adic analogue for a prime $p$ concerns sequences in the p-adic integers $\mathbb{Z}_p$: a sequence $(x_n) \in \mathbb{Z}_p$ is called uniformly distributed if and only if for every $z \in \mathbb{Z}_p$ and $k \in \mathbb{N}$ it holds that
$$\lim_{N \to \infty} \left| \frac{\#\left( D_p(z,1/p^k) \cap \{ x_1,\ldots,x_N \}\right) }{N} - \frac{1}{p^k} \right| = 0,$$
where
$$D_p(z,p^{-k}) := \left\{ x \, : \, \left| x - z \right|_p \leq p^{-k} \right\} \subset \mathbb{Z}_p$$ 
and $|\cdot|_p$ denotes the p-adic absolute value. The term $\tfrac{1}{p^k}$ represents the Haar measure $\mu(D_p(z,p^{-k}))$ and the p-adic definition of uniform distribution thus resembles the real one. The notion was originally introduced in \cite{Cug62}, while a discussion of many important properties of uniformly distributed sequences in $\mathbb{Z}_p$ can also be found e.g. in \cite{KN74}. To quantify the degree of uniformity the $p$-adic discrepancy is a main tool. This notion stems from \cite{Cug62} as well, see also \cite{Som22}. For a sequence  $(x_n)_{n \in \mathbb{N}} \subset \mathbb{Z}_p$ and $N \in \mathbb{N}$, the $p$-adic discrepancy is defined as
$$D_N(x_n) := \sup_{z \in \mathbb{Z}_p, k \in \mathbb{N}} \left| \frac{\#\left( D_p(z,1/p^k) \cap \{ x_1,\ldots,x_N \}\right)}{N} - \frac{1}{p^k} \right|.$$
By the same argument as in the real setting, it is straightforward to see that $\tfrac{1}{N} \leq D_N(x_n) \leq 1$ holds for all sequences $(x_n) \subset \mathbb{Z}_p$ and all $N \in \mathbb{N}$. In fact, there is even a relation between the discrepancy of an arbitrary sequence in $\mathbb{Z}_p$ and the discrepancy of an associated real sequence, see \cite{Mei68} and Theorem~\ref{thm:meijer}. However unlike in  the real setting, compare \cite{Sch72}, there are \emph{sequences} in the p-adic integers with
$$D_N(x_n) = \mathcal{O}\left( \frac{1}{N} \right),$$
i.e. without any additional dependency on $N$ in the numerator as the following result from \cite{Bee69} shows. 
\begin{theorem} \label{thm:disc:kron:quant} Let $x_n = na +b$ for $n \in \mathbb{N}$ with $a,b \in \mathbb{Z}_p$. Then
$$D_N(x_n) = \mathcal{O} \left( \frac{1}{N} \right)$$
if and only if $a$ is a unit, i.e. $a \in \mathbb{Z}_p^\times$.
\end{theorem}
A sequence in $\mathbb{Z}_p$ is therefore called a p-adic low-discrepancy sequence if $D_N(x_n) = \mathcal{O} \left( \frac{1}{N} \right)$ and Theorem~\ref{thm:disc:kron:quant} completely classifies these sequences in the linear case. To the best of the author's knowledge these examples are the only explicitly known examples of low-discrepancy sequences in $\mathbb{Z}_p$.\footnote{The results from \cite{Bee69} or \cite{Mei68} could nonetheless probably be also applied to calculate or estimate the p-adic discrepancy of other examples but it seems like this attempt has not been conducted in the literature. However, we would like to make a more systematic analysis here.} In this paper, we consider sequences $(x_n) = (f(n))$ in $\mathbb{Z}_p$, where $f$ is an arbitrary polynomial  and systematically analyze their p-adic discrepancy. The following theorem may be regarded as an analogue of Theorem~\ref{thm:disc:kron:quant} for arbitrary polynomials. 
\begin{theorem} \label{thm:disc:permutation} Let $f$ be a polynomial. Then $(x_n) = (f(n))$ satisfies $D_N(x_n) = \mathcal{O}\left(\frac{1}{N} \right)$ if and only if $f$ is a permutation polynomial $\mod p$ and $\mod p^2$.
\end{theorem}
Recall that a polynomial $f: \mathbb{Z} \to \mathbb{Z}$ is called a permutation polynomial $\textrm{mod} \ n$ for $n \in \mathbb{N}$ if it is bijection on $\mathbb{Z} / n\mathbb{Z}$ and thus a permutation of the elements in $\mathbb{Z} / n\mathbb{Z}$. It is now not difficult to use Theorem~\ref{thm:disc:kron:quant} to construct further explicit examples of low-discrepancy sequences in $\mathbb{Z}_p$.
\begin{proposition} \label{prop:examples} Let $p$ be a prime. Then $x^p+ax+b$ with $a \in \mathbb{Z}_p^x, a+1 \in \mathbb{Z}_p^x$ and $b \in \mathbb{Z}_p$ is a permutation polynomial $\mod p^n$ for all $n \geq 1$. Thus $(x_n) = (f(n))$ is a low-discrepancy sequence in $\mathbb{Z}_p$.
\end{proposition} 
More generally, N\"obauer gave a constructional (but nonetheless implicit) proof in \cite{Noe66} that for all primes $p$ there are infinitely permutation polynomials $\mod p^n$ for all $n \in \mathbb{N}$. Although the examples from Proposition~\ref{prop:examples} are clearly not covered by Theorem~\ref{thm:disc:kron:quant} and also the reductions of the elements $f(1), f(2), \ldots, f(p^k) \mod p^k$ are different from this case, there is a slightly more subtle connection. Therefore, the examples from Proposition~\ref{prop:examples} cann ot be regarded as entirely new. In fact, the answer to the question whether a polynomial yields a low-discrepancy sequence can be purely read off from two associated polynomials of degree $\leq p-2$.
\begin{theorem} \label{thm:main} Let $f(x) = a_k x^k + a_{k-1} x^{k-1} + \ldots a_1 x^1 + a_0$ be a polynomial and set $a_{k+n} = 0$ for all $n \in \mathbb{N}$. We define the associated polynomials
$$g_1(x) = \sum_{k=0}^{p-2} \left(\sum_{j=0}^\infty  a_{k+j\cdot (p-1)} \right) x^k $$
and
$$g_2(x) = \sum_{k=0}^{p-2}  \left(\sum_{\substack{j=0\\ k+1-j \not \equiv p}}^\infty   (k+1-j) a_{k+1+j\cdot (p-1)} \right) x^k.$$
Then $x_n = f(n)$ is a low-discrepancy sequence if and only if $g_1(n)$ is a permutation polynomial and $g_2(n)$ does not have a solution $g_2(n) \equiv 0 \mod p$
\end{theorem}
Coming back to Proposition~\ref{prop:examples}, we see that $g_1(x) = (a+1)x + b$ which is a permutation polynomial if $a+1 \not \equiv 0 \mod p$. Moreover $g_2(x) = a \not \equiv 0 \mod p$ because $a \in \mathbb{Z}_p^x$. Hence it follows from a result by N\"obauer, see \cite{Noe65}, which we  will mention as Proposition~\ref{prop:criterion_permutation}, that Proposition~\ref{prop:examples} is a direct consequence of Theorem~\ref{thm:main}.\\[12pt]
More specifically the property of being a low-discrepancy sequence in $\mathbb{Z}_3$ can be purely read off from two associated polynomials of the form $ax+b$. An explicit version for this case is formulated in Corollary~\ref{cor:p=3}. Also in the general case we see that only an understanding of the finitely many polynomials of degree $k \leq p-2$ is necessary. If one looks more closely at the construction of the already mentioned infinitely many permutation polynomials ($\textrm{mod} \ p^n$ for all $n \in \mathbb{N}$) according to \cite{Noe66}, then this is achieved by higher degree polynomials which is in line with the assertion from Theorem~\ref{thm:main}.\\[12pt]
By means of Corollary~\ref{cor:p=3} and Theorem \ref{thm:disc:kron:quant} we are thus in the position for any polynomial $f$ to decide if it yields a low-discrepancy sequence $(x_n) = (f(n))$ in $\mathbb{Z}_3$ or not. In other words, the discrepancy of polynomial sequences can be completely classified in this situation. Of course, it would be desirable to extend this classification to arbitrary primes. However, this would demand a complete classification of permutation polynomials up to degree $p-2$, which, of course, can be done by an explicit calculation for a given $p$ but seems to be currently out of reach without performing these explicit calculations, compare \cite{Hou15}.\\[12pt] 
Nonetheless, permutation polynomials up to degree $6$ have been completely classified in the following way: on the class of permutation polynomials an equivalence relation is defined via $f \sim g$ if and only if $f = A \circ g \circ B$ for some affine maps $A,B$ of the form $ax+b$ with $a \not \equiv 0 \mod p$. Under this equivalence Dickson classified \emph{all} permutation polynomials of degree $\leq 5$ and for \emph{all odd} $p$ of degree $6$ in the classic paper \cite{Dic96}. The list was rather recently verified in \cite{SW13}. Based on this, we prove the following result yielding a list of low-discrepancy which are now truly different from Theorem~\ref{thm:disc:kron:quant} in the sense that they can also not be identified by making use of Theorem~\ref{thm:main}.
\begin{theorem} \label{thm:examples_small_degree}
    The only polynomials of degree $2 \leq d \leq 6$ apart from those in Proposition~\ref{prop:examples} such that $(x_n) = (f(n))$ is a low-discrepancy sequence in $\mathbb{Z}_p$ are listed in Table~\ref{tab:low-discrepancy_polynomials}, where $m \in \mathbb{N}$ is arbitrary. 
    \begin{table}[h]
\centering
    \begin{tabular}{c|c}
    $f(x)$ & $p$\\
    \hline
       $x^5 + 2ax^3 +a^2x, a \not \equiv y^2$ & $5$\\
       $x^6 \pm 2x$ & $11$\\
       $x^6 \pm 4x$ & $11$\\
       $x^5+ax^3+5^{-1}a^2x, a \not \equiv 0$ & $5m \pm 2$\\
    \end{tabular}
    \caption{Polynomials of degree $2 \leq d \leq 6$ generating low-discrepancy sequences in $\mathbb{Z}_p$}
    \label{tab:low-discrepancy_polynomials}
\end{table}
\end{theorem}
Here $a \not \equiv y^2$ means that $a$ is not a square $\mod p$ and normalized refers to the leading coefficient being equal to $1$. However, multiplication by units (or more generally composition with affine maps, see above) does not change the property of being a permutation polynomial. Moreover, note that the first polynomial $x^5+2ax^3+a^2x$ is associated to a polynomial of degree $3$ by Theorem~\ref{thm:main}.\\[12pt]
Originally, research on this paper was motivated by the search for examples of sequences with (strong) p-adic Poissonian pair correlations, see \cite{Wei23c} or Definition~\ref{def:PPC} for details. It seemed promising to the author to try to find examples in $\mathbb{Z}_p$ generated by polynomials. Indeed, the class of examples which we constructed in this paper are low-discrepancy sequences and hence have weak p-adic Poissonian pair correlations for all $0 < \alpha < 1$ according to \cite{Wei23c}, Proposition 4.1. However this does not hold for $\alpha = 1$ which would correspond to (strong) p-adic Poissonian pair correlations. This is clear in the case when $f$ is not a permutation polynomial, compare Remark~\ref{rem:uniform}, and is also true in general. Therefore, it is impossible to find an example $(x_n)=(f(n))$ having Poissonian pair correlations as it was originally intended by the author.
\begin{theorem} \label{thm:poissonian} Let $f$ be an arbitrary polynomial. Then the sequence $(x_n) = (f(n))$ does not have p-adic Poissonian pair correlations.
\end{theorem}
This paper is organized as follows: in Section~\ref{sec:pol_disc} we discuss permutation polynomials and their role in p-adic discrepancy theory. Afterwards we explain in Section~\ref{sec:PPC}, why polynomial sequences necessarily fail to have p-adic Poissonian pair correlations. Finally, we briefly discuss as an application how our results can be used to construct (new) real low-discrepancy sequences in Section~\ref{sec:real}.


\section{Permutation Polynomials and Discrepancy} \label{sec:pol_disc}

Let us start by fixing some notation and discussing basic properties of the p-adic numbers. Let $p \in \mathbb{Z}$ be a prime number. For $a = \frac{b}{c}$ with $b,c \in \mathbb{Z} \setminus \{ 0 \}$, let $m$ be the highest possible power with $a = p^m \frac{b'}{c'}$ and $(b'c',p)=1$. Then the p-adic absolute value of $a$ is given by
$$|a|_p:= p^{-m}$$
and the p-adic numbers $\mathbb{Q}_p$ are the completion of $\mathbb{Q}$ with respect to $|\cdot|_p$. The analogue of $[0,1] \subset \mathbb{R}$ are the p-adic integers
$$\mathbb{Z}_p:=\left\{ x \in \mathbb{Q}_p \, : \, \left|x\right|_p \leq 1 \right\}$$
which form a subring of $\mathbb{Q}_p$. They are the closure of $\mathbb{Z}$ in the field $\mathbb{Q}_p$ The field $\mathbb{Q}_p$ is a countable union of copies of $\mathbb{Z}_p$ (as $\mathbb{R}$ is a countable union of copies of $[0,1]$) which is given by
$$\mathbb{Q}_p = \bigcup_{m \geq 0} p^{-m} \mathbb{Z}_p.$$
Finally, the ring of units are the elements  $\mathbb{Z}_p^\times:=\left\{ x \in \mathbb{Z}_p \, : \, \left|x\right|_p = 1 \right\}$, i.e. those which are not divisbile by $p$.\\[12pt]
Before we come to the general theory for sequences $(x_n) = (f(n))$, where $f$ is a polynomial, we turn to two rather simple classes of examples, one positive and one negative. The existence proof relies on Theorem~\ref{thm:disc:permutation} and a proposition which goes back to N\"obauer in \cite{Noe66}.
\begin{proposition} \label{prop:criterion_permutation} A polynomial is a permutation polynomial $\mod p^2$ for a prime $p$ if and only if it is a permutation polynomial $\mod p$ and the congruence $P'(x) \equiv 0 \mod p$ has no solutions.
\end{proposition}
The proof is now a direct consequence of the two mentioned results.
\begin{proof}[Proof of Proposition~\ref{prop:examples}] First we consider the congruence $\mod p$. Note that $x^{p} \equiv x \mod p$ by Fermat's little Theorem. This means that $x^p+ax \equiv (1+a)x \mod p$. Therefore, $f(x)$ is a permutation polynomial if $p \nmid 1+a$. In order to check if $f(x)$ is also a permutation polynomial $\mod p^2$, we only need to show that the derivative of $f'(x) = px^{p-1} + a$ does not have a solution $f'(x) \equiv 0 \mod p$ by Proposition~\ref{prop:criterion_permutation}. Since $f'(x) \equiv a \mod p$ and $a \in \mathbb{Z}_p^x$ this holds automatically.
\end{proof}
Indeed, the argument does not work without the linear term as the following proposition implies although the monic polynomials $x^n$ are standard examples of permutation polynomials $\mod p$ if $(n,p-1) = 1$. 
\begin{proposition} \label{prop:non-example} Let $p$ be an arbitrary prime number. Polynomials of the form $f(x) = ax^n +b$ with $a \in \mathbb{Z}_p^x$ are never permutation polynomials $\mod p$ and $\mod p^2$ if $n > 1$. 
\end{proposition}  
Hence these polynomials cannot yield low-discrepancy sequences $(x_n) = (f(n))$ by Theorem~\ref{thm:disc:permutation}.
\begin{proof} The polynomial $ax^n + b$ is a permutation polynomial if and only if $(n,p-1)=1$ implying that $n$ must be odd. However, then $f'(x) = nax^{n-1}$ cannot be a permutation polynomial because $n-1$ is even and thus $(n-1,p-1) \geq 2.$
\end{proof}
The infinitely many examples of permutation polynomials $\mod p^n$ for all $n \geq 1$ which are obtained by the construction of N\"obauer in \cite{Noe65}, see also Theorem~2.8 in \cite{Nar84}, are constructed by increasing the degree of the polynomial. In the same spirit, Theorem~\ref{thm:main} may not only be regarded as a criterion to check if a given polynomial is a permutation polynomial but also as a possibility to obtain permutation polynomials of arbitrarily large degree.
\begin{proof}[Proof of Theorem~\ref{thm:main}] Since $x^{p-1} \equiv 1 \mod p$, it follows that $x^{j(p-1)+k} \equiv x^k \mod p$. Hence $f(x)$ is a permutation polynomial $\mod p$ if and only if $g_1(n)$ is a permutation polynomial $\mod p$. The derivative of $f(x)$ is
$$f'(x) = \sum_{k=1}^\infty k a_k x^{k-1}$$
which is equivalent $\mod p$ to
$$f'(x) = \sum_{k=1}^\infty [k]_p [a_k]_p x^{k-1}$$
where $[\cdot]_p$ denotes the residue class $\mod p$. By $x^{p-1} \equiv 1 \mod p$, it follows that $f'(x)$ has a root if and only if $g_2(x)$ has a root. Thus the claim follows by Proposition~\ref{prop:criterion_permutation}.
\end{proof}
 In the case $p=3$, Theorem~\ref{thm:main} allows to immediately read off from the coefficients of a given linear polynomial if it yields a low-discrepancy sequence or not. As it settles the situation in the case $p=3$ completely, we formulate this version of Theorem~\ref{thm:main} here explicitly as a corollary.
\begin{corollary} \label{cor:p=3} Let $f(x) = a_k x^k + a_{k-1} x^{k-1} + \ldots a_1 x^1 + a_0$. Define the associated polynomials
$$g_1(x) = (a_0 + a_2 + a_4 +\ldots)  + (a_1 + a_3 + \ldots)  x$$
and
$$g_2(x) = (a_1 + 2a_5 + a_7 + 2a_{11} + \ldots) + (2a_2 + a_4 + 2a_8 + a_{10} +  \ldots) x.$$
Then $x_n = f(n)$ is a low-discrepancy sequence if and only if $g_1(n)$ is a permutation polynomial and $g_2(n)$ does not have a solution $g_2(n) \equiv 0 \mod 3$
\end{corollary}
Moreover, Theorem~\ref{thm:main} shows that the most relevant polynomials to discover are those of small degree. For general $p$ this is a demanding task which has been solved for degree $\leq 6$ in \cite{Dic96} and \cite{SW13} but which is open for higher degrees, compare~\cite{Hou15}. However, this classification suffices to complete the picture for $p \in \{ 5, 7 \}$ as well. We state it here in a form with an extra condition on the coefficient of the linear term which is tailored to our needs below (and which is more compact than the general result).
\begin{theorem} \label{thm:dickson} The combinations of normalized polynomials $f(x)$ and prime numbers $p$ listed in Table~\ref{tab:dickson} are the only permutation polynomials up to degree $6$ with non-vanishing linear coefficient mod any prime number. 
\begin{table}[h]
\centering
    \begin{tabular}{c|c}
    $f(x)$ & $p$\\
    \hline
       $x^3 - ax, a \not \equiv y^2$ & $3$\\
       $x^4 \pm 3x$ & $7$ \\
       $x^5 - ax, a \not \equiv y^4$ & $5$\\
       $x^5 + ax^3 \pm x^2 +3a^2x, a \not \equiv y^2$ & $7$\\
       $x^5 + ax^3 + 5^{-1}a^2x, a \not \equiv 0$ & $5m \pm 2$\\
       $x^5 + ax^3 +3a^2x, a \not \equiv y^2$ & $13$\\
       $x^5 + 2ax^3 +a^2x, a \not \equiv y^2$ & $5$\\
       $x^6 \pm 2x$ & $11$\\
       $x^6 \pm 4x$ & $11$\\
       $x^6 \pm a^2x^3+ax^2\pm5x, a = y^2$ & $11$\\
       $x^6 \pm 4a^2x^3 + ax^2 \pm 4x, a \not \equiv y^2$ & $11$\   
    \end{tabular}
    \caption{List of normalized permutation polynomials up to degree $6$.}
    \label{tab:dickson}
\end{table}
\end{theorem}
Theorem~\ref{thm:examples_small_degree} is now an immediate application of Theorem~\ref{thm:dickson}, Theorem~\ref{thm:disc:permutation} and Proposition~\ref{prop:criterion_permutation}.
\begin{proof}[Proof of Theorem~\ref{thm:examples_small_degree}] Let $f(x) = a_6 x^6 + a_5 x^5 + \ldots + a_1 x + a_0$. Since $f'(n) \equiv 0 \mod p$ may not have a solution, it necessarily holds that $a_1 \not \equiv 0 \mod p$. By the work of Dickson, Theorem~\ref{thm:dickson}, the only permutation polynomials of degree $\leq 6$ and $p$ satisfying this property are listed in Table~\ref{tab:dickson}. 
The first and the third polynomial are already covered by Proposition~\ref{prop:examples} and may therefore be excluded in the following. Moreover, we exclude the case $5m \pm 2$ for a moment. The derivatives of the other polynomials and the answer to the question if a root of these exist, can be found in Table~\ref{tab:derivates} (the expression $\exists x_0(a)$ indicates that a root $x_0$ exists for all $a$ but depends on $a$). 
\begin{table}[h]
    \centering
    \begin{tabular}{c|c|c}
    $f'(x)$ & $p$ & roots\\
    \hline
       $4x^3 + 3$ & $7$ & $1,2, 4$ \\
       $4x^3 - 3$ & $7$ & $3,5, 6$ \\
       $5x^4 + 3ax^2 \pm 2x +3a^2, a \not \equiv y^2$ & $7$ &  $\exists x_0(a)$  \\
       $5x^4 + 3ax^2 +3a^2, a \not \equiv y^2$ & $13$ & $\exists x_0(a)$ \\
       $5x^4 + 6ax^2 +a^2, a \not \equiv y^2$ & $5$& ---\\
       $6x^5 \pm 2$ & $11$& ---\\
       $6x^6 \pm 4$ & $11$& ---\\
       $6x^5 \pm 3a^2x^2+2ax\pm 5, a = y^2$ & $11$& $\exists x_0(a)$ \\
       $6x^5 \pm 12a^2x^2 + 2ax \pm 4, \ a \not \equiv y^2$ & $11$&  $\exists x_0(a)$ \\
    \end{tabular}
    \caption{Existence of roots for the derivatives of Dickson's permutation polynomials}
    \label{tab:derivates}
\end{table}
This already completes the proof in all cases but $p = 5m \pm 2$ by Theorem~\ref{thm:disc:permutation} and Proposition~\ref{prop:criterion_permutation}. So let us assume that $p = 5m \pm 2$. Then $p^3 \equiv 3 \mod 5$ if $p \equiv 2 \mod 5$ and $p^2 \equiv 2 \mod 5$ if $p \equiv 3 \mod 5$. Since the polynomials in the table are not only permutation polynomials for $p = 5m \pm 2$ but also for $p^n = 5m \pm 2$ for some $n \in \mathbb{N}$, see again \cite{Dic96}, we conclude that $f(n) = x^5 + ax^3 + 5^{-1}a^2x, a \not \equiv 0$ is a permutation polynomial for any power (by Hensel's lemma). 
\end{proof}
\section{Poissonian pair correlations} \label{sec:PPC}
In \cite{Wei23c}, the concept of (weak) Poissonian pair correlations was transferred from sequences in the real unit interval $[0,1]$ to sequences in the p-adic integers $\mathbb{Z}_p$. Although it was shown therein that a uniformly distributed random sequence in $\mathbb{Z}_p$ generically has Poissonian pair correlations, no \emph{explicit} example was given. Therefore, the author of this paper started research on the present topic by trying to find such examples of sequences $(x_n) \subset \mathbb{Z}_p$ generated by polynomials with Poissonian pair correlations. Let us first recall the definition.
\begin{definition} \label{def:PPC} For $p$ a prime number, $s \in \mathbb{R}^+_0, 0 < \alpha \leq 1$ and $N \in \mathbb{N}$ define
$$F_{N,\alpha,p}(s) := \frac{1}{N^2} \frac{1}{\mu\left( D_p(0,s/N^\alpha)\right)} \# \left\{ 1 \leq i \neq j \leq N \, : \, \left|x_i-x_j\right|_p \leq \frac{s}{N^\alpha} \right\}.$$
We say that $(x_n) \subset \mathbb{Z}_p$ has weak Poissonian pair correlations for $0 < \alpha \leq 1$ if 
$$\lim_{N \to \infty} F_{N,\alpha,p}(s) = 1$$
for all $s \in \mathbb{R}^+_0$. For $\alpha = 1$, we simply say that $(x_n)$ has Poissonian pair correlations.
\end{definition}
However, it can be easily inferred from the proof of Theorem~\ref{thm:disc:permutation} that the author's original aim cannot be achieved. In order to better see this connection, we moved the proof of Theorem~\ref{thm:disc:permutation} to the current section on Poissonian pair correlations. A first and basic observation which will be used is the fact $f(n+N) \equiv f(n) \mod N$ for all $n, N \in \mathbb{N}$. A closer investigation leads of this property leads to the following proof.
\begin{proof}[Proof of Theorem~\ref{thm:disc:permutation}]
It follows automatically from Hensel's lemma, that a permutation $\mod p^2$ is also a permutation $\mod p^k$ for all $k>2$, so it suffices to show uniform distribution $\mod p$ and $\mod p^2$. Assume that $f$ is a permutation polynomial. Then every disk $D(z,p^{-k})$ contains $\lfloor Np^{-k} \rfloor$ or $\lfloor Np^{-k} \rfloor + 1$ elements. This implies that $D_N(x_n) = \mathcal{O}\left( \frac{1}{N} \right)$. If $f$ is not a permutation polynomial, then there exists a $x_0 \in \left\{ 0, 1, \ldots p^2-1 \right\}$ with $f(n) \not \equiv x_0 \mod p^2$ for all $n \in \mathbb{N}$. Therefore $D(z,p^{-3}) = \emptyset$ and $D_N(x_n) \geq p^{-3}$ for all $N \in \mathbb{N}.$
\end{proof}
\begin{remark} \label{rem:uniform} In fact, the proof of Theorem~\ref{thm:disc:permutation} shows for any polynomial $f$ that the sequence $(x_n) = (f(n))$ either satisfies $D_N(x_n) = \mathcal{O}\left(\frac{1}{N} \right)$ or $D_N(x_0)$ does not converge to $0$. 
\end{remark}
It is now immediate to derive Theorem~\ref{thm:poissonian}.
\begin{proof}[Proof of Theorem~\ref{thm:poissonian}] We may assume that $f$ is a permutation polynomial $\mod p^k$ for all $k \in \mathbb{N}$. However, for $N \leq p^{k}$ the difference $f(n)-f(m)$ cannot be divisible by $p^k$ because then $f(n)-f(m)$ would be in the same residue class $\mod p^k$. Hence $|f(n)-f(m)|_p > p^{-k}$ and $|f(n)-f(m)|_p \leq \frac{s}{N}$ cannot hold for any $1 \leq n \neq m \leq N = p^k$ and $s < 1$. Therefore $(x_n)$ cannot have p-adic Poissonian pair correlations.
\end{proof}
Hence, the question how to find \emph{explicit} sequences possessing Poissonian pair correlation remains open for future research.
\section{An application to real discrepancy theory} \label{sec:real}
Surely, real discrepancy theory has aroused more interest in the literature than its p-adic counterpart. Therefore, we shortly discuss at the end of this paper how our results can be transferred to the real setting. As a gain from this we will obtain (new) real low-discrepancy sequences. Let the $p$-adic expansion of $z \in \mathbb{Z}_p$ be given as 
$$z = \sum_{i=0}^\infty a_i p^i$$
for some coefficients $0 \leq a_i < p$. Then define the function $\varphi_p : \mathbb{Z}_p \to [0,1)$ by
$$\varphi_p(z) = \sum_{i=0}^\infty \left\{ a_i p^{-i-1} \right\},$$
where $\left\{ \cdot \right\}$ denotes the fractional part of a number in $\mathbb{R}$. Furthermore, define for a real sequence $(x_n) \subset [0,1)$ and $N \in \mathbb{N}$ its discrepancy by
$$D_N(x_n) = \sup_{[a,b) \subset [0,1)} \left| \frac{\#\left( [a,b) \cap \{ x_1,\ldots,x_N \}\right)}{N} - (b-a) \right|.$$
Meijer proved in \cite{Mei68} the following theorem.
\begin{theorem} \label{thm:meijer} Let $(x_n) \subset \mathbb{Z}_p$ be an arbitrary sequence and denote its p-adic discrepancy for $N \in \mathbb{N}$ by $\delta_N$. Let $d_N$ be the discrepancy of the corresponding (real) sequence $\varphi_p(x_n) \subset [0,1)$. Then it holds that
$$\delta_N < d_N < \delta_N \left( 2 + \frac{2(p-1)}{\log p} \log \left( \delta_N^{-1} \right) \right).$$    
\end{theorem}
The only application of Theorem~\ref{thm:meijer} given in \cite{Mei68} was the sequence $(x_n) = (ax_n+b)$. More generally, it holds for any sequence $(x_n) = f(n) \subset \mathbb{Z}_p$ generated by an appropriate permutation polynomial $f$ that $\delta_N^{-1} \leq \frac{c}{N}$ for some $c \in \mathbb{R}$. Thus it follows from Theorem~\ref{thm:meijer} that $\varphi_p(x_N)$ has discrepancy of order $\mathcal{O}\left( \frac{\log(N)}{N} \right)$ and is therefore a real low-discrepancy sequence by definition, see \cite{KN74}. Since explicit examples of real low-discrepancy sequences are comparably hard to find this observation presents yet another motivation for the investigations conducted in the present paper.

\bibliographystyle{alpha}
\bibdata{literatur}
\bibliography{literatur}

\textsc{Ruhr West University of Applied Sciences, Duisburger Str. 100, D-45479 M\"ulheim an der Ruhr,} \texttt{christian.weiss@hs-ruhrwest.de}
\end{document}